\documentclass{amsart}
\usepackage[utf8]{inputenc}
\usepackage{amssymb,latexsym}
\usepackage{amsmath}
\usepackage{graphicx}
\usepackage{textcomp}
\usepackage[dvipsnames]{xcolor}

\usepackage{amsthm,amssymb,enumerate,graphicx, tikz}
\usepackage{amscd}
\usepackage{setspace}
\usepackage{comment}
\usepackage{hyperref}
\usepackage{cleveref}

\def\@logofont{\footnotesize}
\def\@setaddresses{\par
  \nobreak \begingroup
  \footnotesize
  \def\author##1{\nobreak\addvspace\bigskipamount}%
  \def\\{\par\nobreak}%
  \interlinepenalty\@M
  \def\address##1##2{\begingroup
    \par\addvspace\bigskipamount\indent
    \@ifnotempty{##1}{(\ignorespaces##1\unskip) }%
    {\scshape\ignorespaces##2}\par\endgroup}%
  \def\curraddr##1##2{\begingroup
    \@ifnotempty{##2}{\nobreak\indent\curraddrname
      \@ifnotempty{##1}{, \ignorespaces##1\unskip}\/:\space
      ##2\par}\endgroup}%
  \def\email##1##2{\begingroup
    \@ifnotempty{##2}{\nobreak\indent\emailaddrname
      \@ifnotempty{##1}{, \ignorespaces##1\unskip}\/:\space
      \ttfamily##2\par}\endgroup}%
  \def\urladdr##1##2{\begingroup
    \def~{\char`\~}%
    \@ifnotempty{##2}{\nobreak\indent\urladdrname
      \@ifnotempty{##1}{, \ignorespaces##1\unskip}\/:\space
      \ttfamily##2\par}\endgroup}%
  \addresses
  \endgroup
}
\renewcommand*\subjclass[2][2010]{%
  \def\@subjclass{#2}%
  \@ifundefined{subjclassname@#1}{%
    \ClassWarning{\@classname}{Unknown edition (#1) of Mathematics
      Subject Classification; using '2000'.}%
  }{%
    \@xp\let\@xp\subjclassname\csname subjclassname@#1\endcsname
  }%
}

\usepackage{graphicx,amsmath,amssymb}
\usepackage{algorithm}
\usepackage{mathdots, comment}
\usepackage[noend]{algpseudocode}

\newtheorem{theorem}{Theorem}[section]

\newtheorem*{theorem*}{Theorem}

\newtheorem{observation}[theorem]{Observation}

\newtheorem{corollary}[theorem]{Corollary}

\theoremstyle{definition}
\newtheorem{definition}[theorem]{Definition}

\theoremstyle{remark}
\newtheorem{remark}[theorem]{Remark}
\newtheorem{example}[theorem]{Example}

\newcommand{\M}{\mathcal{M}}

\newcommand{\I}{\mathcal{I}}

\newcommand{\N}{\mathcal{N}}

\begin{document}
\title{Matchings in matroids over abelian groups, II}


\author[M. Aliabadi, Y. Wu, S. Yermolenko]{Mohsen Aliabadi$^{1,a,*}$ \and Yujia Wu$^{2,b}$ \and Sophia Yermolenko$^{2,c}$}
\thanks{$^1$Department of Mathematics, Clayton State University, 2000 Clayton State Boulevard, Morrow, GA 30260, USA.}
\thanks{$^2$Department of Mathematics, University of California, San Diego, 
9500 Gilman Dr, La Jolla, CA 92093, USA.}

\thanks{$^a$ \url{maliabadisr@clayton.edu}}
\thanks{$^b$ \url{yuw172@ucsd.edu}}
\thanks{$^c$ \url{syermolenko@ucsd.edu}}
\thanks{$^*$ Corresponding Author.}

\begin{abstract}
The concept of matchings originated in group theory to address a linear algebra problem related to canonical forms for symmetric tensors. In an abelian group $(G,+)$, a matching is a bijection $f: A \to B$ between two finite subsets $A$ and $B$ of $G$ such that $a + f(a) \notin A$ for all $a \in A$. A group $G$ has the matching property if, for every two finite subsets $A, B \subset G$ of the same size with $0 \notin B$, there exists a matching from $A$ to $B$. In prior work \cite{Zerbib 0}, matroid analogues of results concerning matchings in groups were introduced and established. This paper serves as a sequel, extending that line of inquiry by investigating sparse paving, panhandle, and Schubert matroids through the lens of matchability. While some proofs draw upon earlier findings on the matchability of sparse paving matroids, the paper is designed to be self-contained and accessible without reference to the preceding sequel. Our approach combines tools from both matroid theory and additive number theory.

\end{abstract}

\maketitle

\section{Introduction}

\subsection{Matchings in abelian Groups and their linear versions}

Throughout this paper, we assume that $G$ is an additive abelian group, unless explicitly stated otherwise. Let $B$ be a finite subset of $G$ that excludes the neutral element $0$. For any subset, $A \subset G$ with the same cardinality as $B$, a \textit{matching} from $A$ to $B$ is a bijection $f: A \to B$ such that for any element $a \in A$, the sum $a + f(a)$ does not belong to $A$. The necessary conditions for a matching from $A$ to $B$ to exist are that $|A| = |B|$ and $0 \notin B$. We say that $G$ has the \emph{matching property} if these are sufficient conditions as well.

The concept of matchings in abelian groups was introduced by Fan and Losonczy \cite{Fan} to generalize a geometric property of lattices in Euclidean space, related to an old linear algebra problem by E. K. Wakeford \cite{Wakeford} concerning canonical forms for symmetric tensors.

Matchings have been extensively studied in the context of groups (see \cite{Alon, Losonczy, Aliabadi 0, Aliabadi 1, Aliabadi 4, Aliabadi 5} for various results and \cite{Hamidoune} for enumerative aspects). A closely related concept is that of matchings between subspaces of a field extension, where linear analogues and comparable results are presented in \cite{Eliahou 2}.

It is worth mentioning that while there are connections between the matching theory discussed in this paper and the matching theory introduced in 1935 by Philip Hall, the two concepts are distinct. Specifically, to address a linear algebra problem posed by Wakeford \cite{Wakeford}, a bipartite graph $\mathcal{G} = (V(\mathcal{G}), E(\mathcal{G}))$ was constructed in \cite{Fan} based on two sets $A\subset \mathbb{Z}^n$ and $B\subset\mathbb{Z}^n$, as follows: Since $A$ and $B$ are possibly non-disjoint, we associate a symbol $x_a$ to each $a \in A$ and a symbol $y_b$ to each $b \in B$. We then define $X = \{x_a \mid a \in A\}$ and $Y = \{y_b \mid b \in B\}$. The nodes of $\mathcal{G}$ are given by the bipartition $V(\mathcal{G}) = X \cup Y$, and there is an edge $e(x_a, y_b) \in E(\mathcal{G})$ joining $x_a \in X$ to $y_b \in Y$ if and only if $a + b \notin A$. 
In graph theory, a \textit{matching} is a collection of edges where no two edges share a common node. A \textit{perfect matching} is a matching that covers all the nodes. For the bipartite graph associated with the pair of sets \( A \) and \( B \), perfect matchings correspond to bijections \( f : A \to B \) satisfying \( a + f(a) \notin A \) for all \( a \in A \). Wakeford's problem thus reduces to determining the existence of a specific type of matching, called an acyclic matching from \( A \) to \( B \).

Note that the classical definition of a perfect matching, from a purely graph-theoretical perspective, only requires that every vertex in the graph be incident to exactly one edge in the matching, essentially, all vertices are covered. In this setting, there is no bijection \( f: A \to B \), nor is the condition \( a + f(a) \notin A \) involved. However, matchings in groups adopt this bijective framework to tackle the linear algebra problem posed by Wakeford.

\subsection{Matchings in Matroids}

A matroidal version of results on matchings is presented in \cite{Zerbib 0}, where matchability is explored for two matroid classes: sparse paving matroids and transversal matroids. Additionally, a partial result concerning paving matroids is obtained. Building on this foundation, we extend the study of matchability to include the following matroid classes:

\begin{itemize}
    \item Panhandle matroids,
    \item Schubert matroids.
\end{itemize}

Our primary results pertaining to sparse paving matroids are Theorems \ref{paving} and \ref{paving general}. For panhandle and Schubert matroids, our main results are addressed in Theorems \ref{Asy Panhandle} and \ref{Asymmetric Schubert}.

We first introduce the concept of a matroid. A \textit{matroid} $M$ is a pair $(E, \mathcal{I})$ where $E = E(M)$ is a finite \textit{ground set} and $\mathcal{I}$ is a family of subsets of $E$, called \textit{independent sets}, satisfying the following conditions:
\begin{itemize}
    \item $\emptyset \in \mathcal{I}$.
    \item If $X \in \mathcal{I}$ and $Y \subset X$, then $Y \in \mathcal{I}$.
    \item \textit{Augmentation property}: If $X, Y \in \mathcal{I}$ and $|X| > |Y|$, then there exists $x \in X\setminus Y$ such that $Y \cup \{x\} \in \mathcal{I}$.
\end{itemize}

A maximal independent set is called a \textit{basis}. It follows from the augmentation property that all bases have the same size. The \textit{basis system} of $M$ is the collection of all bases of the matroid. An element in $E$ is called a \textit{loop} if it belongs to no basis.
A \textit{circuit} of $M$ is a minimal dependent set. The \textit{rank} of a subset $X \subseteq E$ is given by
\[ r_M(X) = r(X) = \max\{|X \cap I| : I \in \mathcal{I}\}. \]
The rank of a matroid $M$, denoted by $r(M)$, is defined as $r_M(E(M))$. A {{\em flat}} in  $M$ is a set $F \subset E(M)$ with the property that adjoining any new element to $F$ strictly increases its rank. A flat of rank $r(M)-1$ is called a {{\em hyperplane}}. \\ 
We say that two matroids $M=(E, \mathcal{I})$ and $N=(E', \mathcal{I'})$ are {\it{isomorphic}} if there is a bijection $\phi:E\rightarrow E'$ such that $X\in \mathcal{I}$ if and only if $\phi(X)\in \mathcal{I'}$. In this case, $\phi$ is called a {\it{matroid isomorphism}} and we use the notation $M\cong N$. The dual matroid \( M^* = (E, \mathcal{I}^*) \) is defined so that the bases in \( M^* \) are exactly the complements of the bases in \( M \). 
Given two matroids $M$ and $N$, their \emph{direct sum}, denoted by $M \oplus N$, is the matroid whose underlying set is the disjoint union of $E(M)$ and $E(N)$, and whose independent sets are the disjoint unions of an independent set of $M$ with an independent set of $N$.
\\

 Let $G$ be an abelian group. We say that $M = (E, \mathcal{I})$ is \textit{matroid over $G$} if $E$ is a subset of $G$. Following \cite{Zerbib 0} we define the concept of matchings in matroids over an abelian group. We assume that all matroids are loopless.

\begin{definition}\label{def, matching matroid}  Let  $(G,+)$ be an abelian group.
\begin{enumerate}
\item  Let $M$ and $N$ be two matroids over $G$ with $r(M)=r(N)=n>0$. Let $\mathcal{M}=\{a_1,\ldots,a_n\}$ and $\mathcal{N}=\{b_1,\ldots,b_n\}$ be ordered bases of $M$ and $N$, respectively. We say $\mathcal{M}$ is {{\em matched}} to $\mathcal{N}$ if $a_i+b_i\notin E(M)$, for all $1\leq i\leq n$. 
    
    \item We say that $M$ is {{\em matched}}  to $N$ if for every basis $\mathcal{M}$ of $M$ there exists a basis $\mathcal{N}$ of $N$ so that $\mathcal{M}$ is matched to $\mathcal{N}$.

 \end{enumerate}
\end{definition}

\textbf{Matchings in abelian groups vs. matchings in matroids.}  
\begin{itemize}
    \item Matchability in the group setting can be viewed as a special case of matchability in the matroid setting. Let \( A \) and \( B \) be two finite nonempty subsets of an abelian group \( G \), each of cardinality \( n \), and assume that \( 0 \notin B \). Define a uniform matroid \( M \cong U_{n,n} \) on the ground set \( A \), and similarly define a uniform matroid \( N \cong U_{n,n} \) on the ground set \( B \). Then \( A \) can be matched to \( B \) in the group-theoretic sense if and only if \( M \) can be matched to \( N \) in the matroid-theoretic sense.
Note that uniform matroids form a very small class of all matroids, and \( U_{n,n} \) is a particularly special case among them.

    \item Borrowing notations from Definition \ref{def, matching matroid}, suppose that \( \mathcal{M} \) is matched to \( \mathcal{N} \) as matroid bases. Then for each \( 1 \leq i \leq n \), we have \( a_i + b_i \notin E(M) \), and in particular, \( a_i + b_i \notin \mathcal{M} \). It follows that the map \( a_i \mapsto b_i \) defines a matching in the group-theoretic sense between the subsets \( \mathcal{M} \) and \( \mathcal{N} \) of \( G \). In this way, our definition of matchings in matroids is compatible with the classical notion of matchings in abelian groups.\\
\end{itemize}
\textbf{Compatibility of matchings in matroids with matchings in the vector space setting.} Note that Definition \ref{def, matching matroid} is also consistent with the notion of matchability in vector spaces over field extensions, as defined by Eliahou and Lecouvey \cite{Eliahou 2}.  Let \( K \subset F \) be a field extension, and let \( A, B \subset F \) be \( n \)-dimensional \( K \)-subspaces of \( F \). Suppose \( \mathcal{A} = \{a_1, \dots, a_n\} \) and \( \mathcal{B} = \{b_1, \dots, b_n\} \) are ordered bases of \( A \) and \( B \), respectively. Then \( \mathcal{A} \) is said to be matched to \( \mathcal{B} \) if
\begin{align}\label{match vector}
    a_i^{-1} A \cap B \subset \langle b_1, \dots, \hat{b}_i, \dots, b_n \rangle \quad \text{for every } 1 \leq i \leq n,
\end{align}

where \( \langle b_1, \dots, \hat{b}_i, \dots, b_n \rangle \) denotes the \( K \)-vector space spanned by \( \mathcal{B} \setminus \{b_i\} \). A vector space \(A \) is matched to $B$ if every basis of \( \mathcal{A} \) of $A$ is matched to a basis of \( \mathcal{B} \) of $B$ in this way.

Now let \( G = F \setminus \{0\} \) be the multiplicative group of \( F \), and define the matroids over \( G \) by \( M = (A \setminus \{0\}, \mathcal{I}) \) and \( N = (B \setminus \{0\}, \mathcal{I}') \), where \( \mathcal{I} \) and \( \mathcal{I}' \) are the families of linearly independent subsets of \( A \) and \( B \), respectively. Then the ranks satisfy
\[
r(M) = \dim_K A = n = \dim_K B = r(N).
\]
Suppose the basis \( \mathcal{A} \) is matched to the basis \( \mathcal{B} \) in the vector space setting. It follows from condition (\ref{match vector}) that \( a_i b_i \notin A \) for all \( i \), and hence \( A \) is also matched to \( B \) in the matroid-theoretic sense.

\section{Preliminaries}\label{sec:prelim}
In this section, we aim to familiarize the reader with essential terms, basic theory of
certain total orders in abelian groups, and matching theory, providing a necessary
background for understanding this paper.

\subsection{Total orders in abelian groups}\label{total order}

 For the remainder of this paper, we adopt the following notation. Given an abelian group \( G \) and two subsets \( A \) and \( B \) of \( G \), we define the \emph{sumset} of \( A \) and \( B \) by
\[
A + B = \{ a + b : a \in A,\, b \in B \}.
\]
If \( A = \{a\} \) is a singleton, then we simply write \( A + B = a + B \).

Given a subset \( A \subset G \), a total order \( \preceq \) on \( A \) is said to be \emph{compatible with the group structure} if, for every \( a, b, c \in A \), the relation \( a \preceq b \) implies \( a + c \preceq b + c \).

We say that a non-zero element $x \in G$ is {\it positive} if $0 \preceq x$, and {\it negative} if $x \preceq 0$. For every $a$ and $b$ in $G$, if $a \preceq b$ and $a \neq b$, we use the notation $a \prec b$.

A subset $A$ of $G$ is called {\em positive} ({\em negative}) if every $x \in A$ is positive (negative). In this case, we write $0 \prec A$ ($A \prec 0$). If $A$ is a finite subset of $G$, we say that $x \in A$ is the {\em maximum} ({\em minimum}) element of $A$ if $a \preceq x$ ($x \preceq a$) for every $a \in A$, and denote it by $\text{max}(A)$ ($\text{min}(A)$).

Additionally, if $A \subset G$, then $A^+ = \{a \in A \mid 0 \prec a\}$. The set $A^-$ is defined similarly.

 In \cite{Levi} Levi proved the following:
\begin{theorem}\cite{Levi} \label{linearly ordered groups}
An abelian group $G$ has a total order that is compatible with the group structure if and only if it is torsion-free.
\end{theorem}

When $G$ is not torsion-free, it is possible to equip sufficiently small subsets of $G$ with a total order that is compatible with the group structure of $G$. This can be achieved by utilizing a \textit{rectification principle} due to Lev \cite{Lev}, which asserts that a sufficiently small subset of $G$ can be embedded in the integers while preserving certain additive properties. Utilizing this tool leads to the following result proved in \cite{Zerbib 0}. 
Given an abelian group $G$, denote the smallest cardinality of a non-zero finite subgroup of $G$ by $p(G)$. If $G$ has no finite non-trivial subgroup, we assume that $p(G) = \infty$.

\begin{theorem}\label{rectification}
Let $G$ be an abelian group with the non-trivial torsion subgroup. Let $p=p(G)$ and $A$ be a subset of $G$ with $|A|\leq \lceil{\log_2 p\rceil}$. Then there exists a total ordering 
 $\preceq$ on $A$ that is compatible with  $G$. 
\end{theorem}


\subsection{Matching theory}

We borrow the following theorem from \cite{Aliabadi 4}, which addresses the matchability of sufficiently small subsets in the group setting. 

\begin{theorem}\label{matchable sets, p(G)}

Let $G$ be an abelian group and $A$ and $B$ be finite subsets of $G$ with $|A|=|B|=n<p(G)$, and $0\notin B$. Then there is a matching in the group sense from $A$ to $B$.  
\end{theorem}
The following theorem provides a necessary and sufficient condition for matching subsets of an abelian group in the symmetric case:
\begin{theorem}\cite{Losonczy}\label{symmetric matching}
Let $G$ be an abelian group and let $A$ be a nonempty finite subset of $G$.
Then there is a matching from $A$ to itself if and only if $0 \notin A$.
\end{theorem}

 In the next two sections, we will state and prove our main results. Throughout, we may assume that  $M = (E, \mathcal{I})$ and $N=(E', \mathcal{I'})$ are two matroids over an abelian group $(G,+)$ with neutral element $0$. Any unexplained matroid terminology used here will follow Oxley \cite{Oxley}. We assume that $[n]=\{1,2,\cdots,n\}$ for every positive integer $n$.\\
\section{Sparse paving matroids}

 A matroid $M=(E,\I)$ of rank $n$ is said to be a {{\em paving matroid}} if all of its circuits are of size at least $n$. If \( M \) and \( M^* \) are both paving matroids, then \( M \) is called a {{\em sparse paving matroid}}.
For integers $0\leq n \leq m$, let $E$ be an $m$-element set and let $\mathcal{I}$ be the collection of subsets of $E$ with at most $n$ elements. Then $M=(E, \mathcal{I}) $ is a matroid with rank $n$, called the {\em uniform matroid} and denoted by  $U_{n,m}$. Uniform matroids are a very special class of sparse paving matroids.



A {\it{$d$-partition}} of a set $E$ is a collection $\mathcal{S}$ of subsets of $E$, all of which are of size at least $d$ so that 
every $d$-subset of $E$ is contained in exactly one of the sets in  $\mathcal{S}$.
 The set $\mathcal{S} = \{E\}$ is a $d$-partition for any $d$, which we call a {\it{ trivial
$d$-partition.}}
The connection between paving matroids and d-partitions is given by the following result whose proof can be found in \cite{Oxley}.
\begin{theorem} \cite{Oxley}\label{(n-1) partitions}
If $M$ is a paving matroid of rank $n\geq 2$, then its hyperplanes form
a non-trivial $(n-1)$-partition of $E(M)$. That is, the intersection of any two hyperplanes in $M$ is of size at most $n-2$.
\end{theorem}

In the following theorem, we provide a sufficient condition for matchability of sparse paving matroids over an abelian torsion-free group.
\begin{theorem}\label{paving}
Let $G$ be a torsion-free group equipped with a total order $\preceq$ compatible with its group structure. Let $M$ and $N$ be two matroids over $G$ both of the same rank $n$. Assume further that all the following conditions hold:
\begin{enumerate}
\item
$E(M)$ and $E(N)$ are both positive,
\item
$N$ is a sparse paving matroid,
\item
$|E(M)|=|E(N)|=n+1$,
\item
$x\preceq ny$, where $x=\max(E(M))$, $y=\min(E(N))$ and $ny=y+\underbrace{\cdots}_{n-\text{times}}+y.$
\end{enumerate}
Then $M$ is matched to $N$.
\end{theorem}
\begin{proof}
The proof is straightforward in case $n=1$. So we assume that $n>1$.
Since $|E(M)|=|E(N)|$ and $0\notin E(N)$, it follows from Theorem \ref{matchable sets, p(G)} that there exists a matching  $f:E(M)\to E(N)$ in the group sense. Let $E(M)=\{a_1,\ldots,a_{n+1}\}$. Let $\mathcal{M}=\{a_1,\ldots,a_n\}$ be a basis for $M$. Set $b_i=f(a_i)$, $1\leq i\leq n$, and $\mathcal{N}=\{b_1,\ldots,b_n\}$. If $\mathcal{N}$ is a basis for $N$, then $\mathcal{M}$ is matched to $\mathcal{N}$ in the matroid sense, and we are done. So suppose $\mathcal{N}$ is not a basis for $N$. Since $|E(N)|=n+1$ and $|\mathcal{N}|=n$, then $\mathcal{N}$ is a hyperplane. Define $\mathcal{N}_i=\left(\mathcal{N}\setminus\{b_i\}\right)\cup \{b_{n+1}\}$, where $\{b_{n+1}\}=E(N)\setminus \mathcal{N}$. Note that $|\mathcal{N}_i\cap \mathcal{N}|=n-1$. Since $N$ is a sparse paving matroid, by Theorem \ref{(n-1) partitions} $\mathcal{N}_i$'s are not hyperplanes.  This together with the fact that $r_N(\mathcal{N}_i)\geq n-1$ implies that each $\mathcal{N}_i$ is a basis.
We now claim that there exists $i\in[n]$ such that $a_i+b_{n+1}\notin E(M)$. If not, we have 
\begin{align*}
\mathcal{M} + b_{n+1} \subset E(M).
\end{align*}
Therefore, $\sum_{i=1}^n (a_i + b_{n+1}) \preceq \sum_{i=1}^{n+1} a_i$, as $E(M)$ and $E(N)$ are both positive. Since $|\mathcal{M} + b_{n+1}| = n < n+1 = |E(M)|$, $\mathcal{M} + b_{n+1}$ is a proper subset of $E(M)$. Consequently, $\sum_{i=1}^n (a_i + b_{n+1}) \neq \sum_{i=1}^{n+1} a_i$.

This implies $nb_{n+1} \prec a_{n+1}$, entailing $ny \prec x$. This is a contradiction.

Assume that $i$ is obtained as in the claim. Then $\mathcal{M}$ is matched to $\mathcal{N}_i$ in the matroid sense through $a_j\rightarrow b_j$ for $j\neq i$ and $a_i\rightarrow b_{n+1}$, as needed.
\end{proof}

By utilizing a rectification principle (Theorem \ref{rectification}), which guarantees the existence of a suitable total order for sufficiently small subsets of an abelian group \( G \), Theorem \ref{paving} can be extended as stated in Theorem \ref{paving general}, replacing torsion-free abelian groups with general abelian groups. However, this generalization requires the additional condition:  
\begin{align}\label{total small}
|E(M) \cup E(N) \cup (E(M) + E(N)) \cup \widehat{nE(N)} \cup \{0\}| < \lceil \log_2(p(G)) \rceil,
\end{align}
where $\widehat{nE(N)}=\{ny|\hspace{0.2cm} y\in E(N)\}.$

\begin{theorem}\label{paving general}
Let $G$ be an abelian group. Let $M$ and $N$ be two matroids over $G$ both of the same rank $n$. Assume further that all the following conditions hold:
\begin{enumerate}
\item
$E(M)$ and $E(N)$ are both positive,
\item
$N$ is a sparse paving matroid,
\item
$|E(M)|=|E(N)|=n+1$,
\item
$x\preceq ny$, where $x=\max(E(M))$, $y=\min(E(N))$ and $ny=y+\underbrace{\cdots}_{n-\text{times}}+y$,
\item $n<\max\{2,\frac{-5+\sqrt{5+4\lceil{\log_2 (p(G))\rceil}}}{2}\}.$
\end{enumerate}
Then $M$ is matched to $N$.
\end{theorem}
\begin{proof}
The proof closely follows the argument for Theorem \ref{paving}, with the key difference being that condition (5) establishes equation (\ref{total small}). This allows us to apply Theorem \ref{rectification}, ensuring the existence of a total order compatible with the group structure for sufficiently small subsets of \( G \).
\end{proof}

\begin{remark}
  It is worth mentioning that condition (5) in Theorem \ref{paving general} implies equation \ref{total small}, which provides the needed total order for matroids with small enough ground sets. In other words, we have:
   \begin{align*}
        \ |E(M) \cup E(N) \cup (E(M) + E(N)) \cup \widehat{nE(N)} \cup\{0\} |\leq\\
        |E(M)|+|E(N)|+|E(M)| |E(N)|+ |\widehat{nE(N)}| +|\{0\}| \leq\\
        (n+1)+(n+1)+(n+1)^2+(n+1)+1=\\
        n^2+5n+5<\lceil \log_2 (p(G)) \rceil.
   \end{align*}

\end{remark}
\begin{remark}
    It is worth noting that matchability for sparse paving matroids is examined in \cite{Zerbib 0} as Theorem 2.11, albeit in a slightly different context where condition (4) in Theorem \ref{paving general} is replaced by the condition $\max E(M) \preceq n \max E(N)$. Importantly, neither condition implies the other. Moreover, the proof methods for these theorems are distinct from one another.

\end{remark}
\begin{remark}
Note that the matroid \( N \) in Theorems \ref{paving} and \ref{paving general} is a sparse paving matroid of rank \( n \) on a ground set of size \( n+1 \). Up to isomorphism, there are only two such matroids: \( U_{n,n+1} \) and \( U_{n-1,n} \oplus U_{1,1} \). In other words, a non-uniform sparse paving matroid of rank \( n \) on a ground set of size \( n+1 \) has exactly one circuit-hyperplane, and it becomes uniform after a single relaxation in the circuit-hyperplane sense. With this in mind, it might be tempting to split the arguments of Theorems \ref{paving} and~\ref{paving general} into the two cases \( U_{n,n+1} \) and \( U_{n-1,n} \oplus U_{1,1} \). While the case of \( U_{n,n+1} \) follows immediately from Proposition 2.9 in \cite{Zerbib 0}, the argument for the case \( U_{n-1,n} \oplus U_{1,1} \) remains more intricate.

\end{remark}
\begin{remark}
    In general, within matching theory—even in the group setting, where the structure is relatively simpler—the existence of a total order compatible with the group operation is essential. For instance, in \cite{Alon}, which addresses the existence of a specific family of matchings called acyclic matchings in groups, the proof relies heavily on the presence of a well-behaved total order on $\mathbb{Z}^n$. Unfortunately, such a feature is available only in torsion-free abelian groups \cite{Levi}. Since our goal is to establish results in a more general setting, we instead employ the rectification principle Theorem \ref{rectification}, which ensures the existence of a suitably nice total order for sufficiently small subsets of the group. This is formalized through (\ref{total small}), which we use to guarantee such an order.
\end{remark}
\section{Panhandle matroids and Schubert matroids over Abelian Groups}

\textbf{Panhandle matroids:} Let $n \leq s < m$ be nonnegative integers. The \textit{panhandle matroid} $\mathcal{P}_{n,s,m}$ is a rank $n$ matroid on the ground set $[m]$ with the basis system
$$
\mathcal{B} = \mathcal{B}_{n,s,m} = \left\{ B \in \binom{[m]}{n} \ \Big\vert \ |B \cap [s]| \geq n - 1 \right\}.
$$
Panhandle matroids were first introduced in \cite{Hanley}. It is observed that $\mathcal{B}$ is a special case of lattice path matroids. Notable special cases include the minimal matroids described in \cite{Ferroni 0} and uniform matroids.

We now extend the definition of panhandle matroids to abelian groups. In this context, it is crucial to have a total order that is compatible with the group structure. Such a total order can be obtained in torsion-free groups by invoking Theorem \ref{linearly ordered groups}. Alternatively, for sufficiently small subsets of a general abelian group, Theorem \ref{rectification} can be applied. Therefore, we may assume that either $G$ is torsion-free or $m < \lceil \log_2(p(G)) \rceil$ when $G$ has a nonzero torsion element.

Let $\preceq$ be a total order on $(G, +)$ that is compatible with its structure. Let $m$ be a positive integer. For any nonzero $a \in G$, we define
\[
[m]_{(a,G)} = \{a, 2a, \dots, ma\},
\]
where $ma = a + \underbrace{\cdots}_{m \text{ times}} + a$.

Let $n \leq s < m$ be nonnegative integers. The \textit{extended panhandle matroid} $\mathcal{P}_{n,s,m}(a,G)$ is the rank $n$ matroid on the ground set $[m]_{(a,G)}$ with basis system
\[
\mathcal{B}_G = \mathcal{B}_{n,s,m}(a,G) = \left\{ B \in \binom{[m]_{(a,G)}}{n} \ \Big\vert \ |B \cap [s]_{(a,G)}| \geq n - 1 \right\}.
\]

We abbreviate $[m]_{(a,G)}$, $\mathcal{B}_{(n,s,m)}(a,G)$, and $\mathcal{P}_{n,s,m}(a,G)$ to $[m]_a$, $\mathcal{B}_{(n,s,m)}(a)$, and $\mathcal{P}_{n,s,m}(a)$ when $G$ is clear from the context.\\
\begin{example}\label{Panhandle example}
  We compute \(\mathcal{P}_{3,4,5}\left((2,-1,0), \mathbb{Z}^3\right)\). Note that we consider the lexicographic order on the additive group $\mathbb{Z}^3$ here.
 The ground set is:
\[
[5]_{\left((2, -1, 0), \mathbb{Z}^3\right)} = \left\{(2,-1,0), \, (4,-2,0), \, (6,-3,0), \, (8,-4,0), \, (10,-5,0)\right\}.
\]

This is a matroid of rank 3 and the basis system is:
\[
\mathcal{B}_{3,4,5}\left((2,-1,0), \mathbb{Z}^3\right) = \left\{ B \in \binom{[5]_{\left((2, -1, 0), \mathbb{Z}^3\right)}}{3} \;\middle|\; \left| B \cap [4]_{\left((2, -1, 0), \mathbb{Z}^3\right)} \right| \geq 2 \right\}
\]

\[
= \left\{
\begin{aligned}
&\left\{(2,-1,0), \, (4,-2,0), \, (6,-3,0)\right\}, \\
&\left\{(2,-1,0), \, (6,-3,0), \, (8,-4,0)\right\}, \\
&\left\{(2,-1,0), \, (4,-2,0), \, (8,-4,0)\right\}, \\
&\left\{(4,-2,0), \, (6,-3,0), \, (8,-4,0)\right\}, \\
&\left\{(2,-1,0), \, (4,-2,0), \, (10,-5,0)\right\}, \\
&\left\{(2,-1,0), \, (6,-3,0), \, (10,-5,0)\right\}, \\
&\left\{(2,-1,0), \, (8,-4,0), \, (10,-5,0)\right\}, \\
&\left\{(4,-2,0), \, (6,-3,0), \, (10,-5,0)\right\}, \\
&\left\{(4,-2,0), \, (8,-4,0), \, (10,-5,0)\right\}, \\
&\left\{(6,-3,0), \, (8,-4,0), \, (10,-5,0)\right\}
\end{aligned}
\right\}
.\]
 
\end{example}

\textbf{Schubert matroids:} Let $m$ and $n$ be two positive integers with $n \leq m$. Let $S$ be an $n$-subset of $[m]$. The \textit{Schubert matroid} $SM_m(S)$ is the matroid with ground set $[m]$ and bases
\[
\{T \subset [m] : T \leq S\},
\]
where $T \leq S$ means that $|T| = |S|$ and the $i$-th smallest element of $T$ does not exceed that of $S$ for $1 \leq i \leq n$. Schubert matroids were introduced in \cite{Crapo} to prove that there are at least $2^m$ nonisomorphic matroids on $m$ elements. 

In what follows, we define Schubert matroids over abelian groups. Similar to the panhandle matroid case, we assume that $G$ (or small enough subsets of $G$) is equipped with a total ordering $\preceq$ that is compatible with its structure. In particular, for a group containing nonzero torsion elements, we assume that $m$ and $n$ are two positive integers with $n \leq m < \lceil \log_2 (p(G)) \rceil$. Let $a \in G^+$ and $S$ be an $n$-subset of $[m]_a$.

The \textit{extended Schubert matroid:} $SM_m(a,G,S)$ is the rank $n$ matroid whose ground set is $[m]_a$ and whose basis system is
\[
\{T \subset [m]_a : T \preceq S\},
\]
where $T \preceq S$ means that $|T| = |S|$ and the $i$-th smallest element of $T$ does not exceed that of $S$ for $1 \leq i \leq n$. Note that $SM_m(1,\mathbb{Z},S) = SM_m(S)$. The negative counterpart of Schubert matroids may be defined in a similar way if $a \in G^-$. We shorten $SM_m(a,G,S)$ to $SM_m(a,S)$ when $G$ is clear from the context.

\begin{example}\label{Schubert example}
   We compute $ SM_{5}\left((2, -1, 0), \mathbb{Z}^3, S\right) $, where 
   $$ 
   S = \left\{(2, -1, 0), (4, -2, 0), (10, -5, 0)\right\}.
   $$

   The ground set is given by:
   \[
   [5]_{\left((2, -1, 0), \mathbb{Z}^3\right)} = \left\{(2, -1, 0), \, (4, -2, 0), \, (6, -3, 0), \, (8, -4, 0), \, (10, -5, 0)\right\}
   \]

   The basis system is defined as:
   \[
   \mathcal{B} = \left\{ T \subseteq [5]_{(2, -1, 0)} \mid T \preceq \left\{(2, -1, 0), (4, -2, 0), (10, -5, 0)\right\} \right\}
   \]

   \[
   = \left\{
   \begin{aligned}
   &\left\{(2, -1, 0), \, (4, -2, 0), \, (6, -3, 0)\right\}, \\
   &\left\{(2, -1, 0), \, (4, -2, 0), \, (8, -4, 0)\right\}, \\
   &\left\{(2, -1, 0), \, (4, -2, 0), \, (10, -5, 0)\right\}
   \end{aligned}
   \right\}.
   \]
   Therefore, $ SM_{5}\left((2, -1, 0), \mathbb{Z}^3, S\right) $ has exactly three bases.
\end{example}

The following observation characterizes uniform Schubert matroids:
\begin{observation}\label{uniform Schubert}
The extended Schubert matroid $SM_m(a,S)$ is isomorphic to the uniform matroid $U_{n,m}$ if and only if $S = \{(m-n+1)a, \dots, ma\}$.
\end{observation}

\begin{proof}
This follows immediately from the definition of an extended Schubert matroid.
\end{proof}

The proofs of Theorems \ref{Asy Panhandle} and \ref{Asymmetric Schubert} rely on the existence of a total order on 
\begin{align}\label{total order existence}
    [m]_a \cup ([m]_a+[m]_a)\cup\{0\} = [2m]_a\cup\{0\},
\end{align}
that is compatible with the group structure of $G$. For the existence of such orders, see Section \ref{total order}. In particular, Theorem \ref{linearly ordered groups} and Theorem \ref{rectification} imply that such an order exists whenever $G$ is torsion-free, or when (\ref{total order existence}) is as small as:
\begin{align}\label{total order equ}
    m < \frac{\lceil \log_2 (p(G)) \rceil}{2}.
\end{align}

In our results, for simplicity, we say {\it{$m$ is sufficiently small}} if (\ref{total order equ}) holds. Note that in an abelian torsion-free group $G$, since $\lceil \log_2 (p(G)) \rceil = \infty$, every $m \in \mathbb{N}$ satisfies the inequality in (\ref{total order equ}). Consequently, every subset of $G$ can be regarded as sufficiently small.\\
We are now ready to state and prove our results on matchability between panhandle and Schubert matroids in asymmetric settings, as presented in Theorems~\ref{Asy Panhandle} and~\ref{Asymmetric Schubert}.
\begin{theorem} \label{Asy Panhandle}
Let $G$ be an abelian group and $n \leq s, s' < m$ be nonnegative integers. Let $m$ be sufficiently small. Let $a \in G$ be nonzero and $S \subset [m]_a$ be an $n$-subset. Let $M = \mathcal{P}_{n,s,m}(a)$ or $M = SM_m(a,S)$. Then $M$ is matched to $\mathcal{P}_{n,s',m}(a)$ if and only if $s' = m-1$.
\end{theorem}

\begin{proof}

First, assume that $M$ is matched to $\mathcal{P}_{n,s',m}(a)$. From the structure of $M$, it follows that the $n$-subset $\mathcal{M} = \{a, 2a, \dots, na\}$ of $[m]_a$ is a basis for $M$. Consequently, it must be matched to a basis $\mathcal{N} = \{b_1, b_2, \dots, b_n\}$ of $\mathcal{P}_{n,s',m}(a)$, where $b_1 \prec b_2 \prec \dots \prec b_n$.

Assume that $a$ is mapped to $b_{i_1}$ for some index $i_1 \in [n]$. By the matchability of bases in the matroid sense, we have $ma \prec a + b_{i_1}$. On the other hand, we can express this as:
\[
a + b_{i_1} \preceq a + ma = (m+1)a,
\]
which implies that $b_{i_1} = ma$, leading to the conclusion that $i_1 = n$. Therefore, we find that $a$ is mapped to $b_n = ma$.

Next, assume that $2a$ is mapped to $b_{i_2}$. Similarly, we can deduce that $2a$ must be mapped to $b_{n-1} = (m-1)a$. Continuing in this manner, we conclude that the basis $\mathcal{N}$ takes the form $\{(m-n+1)a, \dots, ma\}$. Since $\mathcal{N}$ is a basis of $\mathcal{P}_{n,s',m}(a)$, we must have $|\mathcal{N} \cap [s']_{(a)}| \geq n-1$. This implies that $s' = m-1$.

Conversely, assume that $s' \neq m-1$. Thus, we have $s' < m-1$. Consider the basis $\mathcal{M} = \{a, 2a, \dots, na\}$ of $M$. We claim that $\mathcal{M}$ cannot be matched to any basis of $\mathcal{P}_{n,s',m}(a)$. Indeed, if $\mathcal{M}$ is matched to a basis $\mathcal{N} = \{b_1, \dots, b_n\}$ of $\mathcal{P}_{n,s',m}(a)$ with $b_1 \prec b_2 \prec \dots \prec b_n$, then it follows that $a$ must be mapped to $b_n = ma$.

Now, assume that $2a$ is mapped to $b_i$ for some $i \in [n-1]$. Then we have:
\[
ma \prec 2a + b_i \preceq 2a + s'a = (2 + s')a \preceq ma,
\]
which leads to a contradiction. Therefore, we conclude that $\mathcal{M}$ cannot be matched to any basis of $\mathcal{P}_{n,s',m}(a)$.

\end{proof}
\begin{corollary}[Symmetric matchability for panhandle matroids]\label{sym pan}
For sufficiently small $m$, the matroid $\mathcal{P}_{n,s,m}(a)$ is matched to itself.
\end{corollary}

\begin{proof}
This result follows directly from Theorem \ref{Asy Panhandle}.
\end{proof}

\begin{remark}
It is noteworthy that Corollary \ref{sym pan} serves as the panhandle matroid analogue of Theorem 2.1 in \cite{Zerbib 0}, which establishes symmetric matchability for sparse paving matroids.
\end{remark}

\begin{example}
Consider the matroids $\mathcal{P}_{3,4,5}((2,-1,0), \mathbb{Z}^3)$ and $SM_5((2,-1,0), \mathbb{Z}^3, S)$ as discussed in Examples \ref{Panhandle example} and \ref{Schubert example}. By Theorem \ref{Asy Panhandle}:
\begin{itemize}
    \item The matroid $\mathcal{P}_{3,4,5}((2,-1,0), \mathbb{Z}^3)$ is matched to itself.
    \item The matroid $\mathcal{P}_{3,4,5}((2,-1,0), \mathbb{Z}^3)$ is matched to $SM_5((2,-1,0), \mathbb{Z}^3, S)$.
\end{itemize}
\end{example}

\begin{theorem}\label{Asymmetric Schubert}

Let $G$ be an abelian group, and let $n \leq s < m$ be positive integers with $m$ sufficiently small. Let $a \in G$ be a nonzero element. Consider two subsets $S$ and $S'$ of $[m]_a$ such that $|S| = |S'| = n$. Assume that $M = SM_m(a,S)$ or $M = \mathcal{P}_{n,s,m}(a)$. Then, $M$ is matched to $SM_m(a,S')$ if and only if $SM_m(a,S') \cong U_{n,m}$.

\end{theorem}
\begin{proof}

First, assume that $M$ is matched to $SM_{m}(a, S')$. Without loss of generality, we can assume that $a$ is positive. Let $S' = \{b_1, \dots, b_n\}$, where $b_1 \prec \dots \prec b_n$. Consider the $n$-subset $\mathcal{M} = \{a, 2a, \dots, na\}$ of $[m]_a$. By the definition of $\mathcal{P}_{n, s, m}(a)$ (or $SM_{m}(a, S)$), $\mathcal{M}$ is a basis for $M$. Therefore, it must be matched to a basis $\mathcal{N} = \{a_1, \dots, a_n\}$ of $SM_m(a, S')$. 

Assume that $a$ is matched to $a_{i_1}$ for some index $i_1 \in [n]$. Then we have:
\begin{align}\label{Schubert 1}
    ma \prec a + a_{i_1} \preceq a + b_{i_1} \preceq a + (m - n + i_1)a.
\end{align}
This implies that $m < 1 + m - n + i_1$. Since $i_1 \in [n]$, it follows that $i_1 = n$. Rewriting (\ref{Schubert 1}), we obtain
$$
ma \prec a + a_n \preceq a + b_n \preceq (m + 1)a,
$$
which implies that $a_{i_1} = a_n = b_n = ma$.\\

Now, assume that $2a$ is mapped to $a_{i_2}$, for some $i_2 \in [n-1]$. Similarly, we can argue that $i_2 = n - 1$, and $a_{n - 1} = b_{n - 1} = (m - 1)a$. Continuing in this manner, we deduce that $b_i = (m - i + 1)a$, implying $S' = \{m - n + 1, \dots, ma\}$. Thus, by Observation \ref{uniform Schubert}, we conclude that $SM_{m}(a, S') \cong U_{n, m}$.\\
Conversely, assume that $SM_{m}(a, S') \cong U_{n, m}$. Since $0 \notin [m]_a$, it follows from Theorem \ref{symmetric matching} that $[m]_a$ is matched to itself in the group sense. That is, there exists a bijection $f: [m]_a \rightarrow [m]_a$ such that $x + f(x) \notin [m]_a$ for any $x \in [m]_a$. Let $\M = \{a_1, \dots, a_n\}$ be a basis for $M$, set $b_i = f(a_i)$ for $i \in [n]$ and $\N = \{b_1, \dots, b_n\}$. Since $SM_{m}(a, S')$ is a uniform matroid with rank $n$, $\N$ is a basis for $SM_{m}(a, S')$. $\M$ is matched to $\N$ in the matroid sense.
\end{proof}




\begin{corollary}[Symmetric matchability for Schubert matroids]
For sufficiently small $m$, the rank $n$ matroid $SM_m(a,S)$ is matched to itself if and only if $SM_m(a,S) \cong U_{n,m}$.
\end{corollary}
\begin{proof}
    It is immediate from Theorem \ref{Asymmetric Schubert} along with Observation \ref{uniform Schubert}.
\end{proof}
\begin{example}
    Consider the matroids $\mathcal{P}_{3,4,5}((2,-1,0), \mathbb{Z}^3)$ and $SM_5((2,-1,0), \mathbb{Z}^3, S)$ as discussed in Examples \ref{Panhandle example} and \ref{Schubert example}. By Theorem \ref{Asymmetric Schubert}:
\begin{itemize}
    \item The matroid $SM_5((2,-1,0), \mathbb{Z}^3, S)$ is not matched to itself.
    \item The matroid $\mathcal{P}_{3,4,5}((2,-1,0), \mathbb{Z}^3)$ is not matched to $SM_5((2,-1,0), \mathbb{Z}^3, S)$.
\end{itemize}
\end{example}
\begin{example}
    Consider the matroids $\mathcal{P}_{3,4,5}((2,-1,0), \mathbb{Z}^3)$ and $SM_5((2,-1,0), \mathbb{Z}^3, T)$, where $T=\left\{ \ (6, -3, 0), \, (8, -4, 0), \, (10, -5, 0)\right\}$. By Theorem \ref{Asymmetric Schubert}:
    \begin{itemize}
        \item The matroid $SM_5((2,-1,0), \mathbb{Z}^3, T)$ is matched to itself.
        \item The matroid $\mathcal{P}_{3,4,5}((2,-1,0), \mathbb{Z}^3)$ is matched to $SM_5((2,-1,0), \mathbb{Z}^3, T)$.
    \end{itemize}
\end{example}

\section*{Acknowledgements} 
We are grateful to Shira Zerbib for her valuable suggestions and many insightful conversations. We also thank the referees for their helpful comments and constructive feedback.\\


\textbf{Data sharing:} Data sharing is not applicable to this article as no datasets were generated or analyzed.\\
\textbf{Conflict of interest:} To our best knowledge, no conflict of interests, whether of financial or personal nature, has influenced the work presented in this article.


\begin{thebibliography}{10}

\bibitem{Aliabadi 0}
M. Aliabadi, K. Filom. Results and questions on matchings in abelian groups and vector subspaces of fields. \textit{J. Algebra}. 598 (2022) 85--104.
\bibitem{Aliabadi 1}
M. Aliabadi, M. V. Janardhanan. On local matching property in groups and vector spaces. \textit{Australas. J. Combin.} 70 (2018), 75–85.
\bibitem{Aliabadi 4}
M. Aliabadi, J. Kinseth, C. Kunz, H. Serdarevic, C. Willis.  Conditions for matchability in groups and field extensions. \textit{Linear and Multilinear Algebra.}  71 (2023), no. 7, 1182–1197.
\bibitem{Aliabadi 5}
M. Aliabadi, P. Taylor, Classifying abelian groups through acyclic matchings, To appear in \textit{Ann. Combin}.

\bibitem{Zerbib 0}
M. Aliabadi, S. Zerbib. Matchings in matroids over abelian groups. \textit{J. Algebraic. Combinatorics}. 59 (2024), no. 4, 761–785.
\bibitem{Alon} 
 N. Alon, C. K. Fan, D. Kleitman, and J. Losonczy. Acyclic matchings. \textit {Adv. Math}.,
122(2):234–236, 1996.
 \bibitem{Crapo}
H. Crapo, Single-element extensions of matroids, \textit{Journal of research, National Bureau of
Standards} 69B (1965), 55–66.
\bibitem{Eliahou 2}
 S. Eliahou, C. Lecouvey, Matching subspaces in a field extension.\textit{J. Algebra} 324 (2010), 3420-
3430.
\bibitem{Fan}
C. K. Fan, J. Losonczy, Matchings and canonical forms for symmetric tensors. \textit {Adv. Math.} 117 (1996), no. 2, 228–238.
\bibitem{Ferroni 0}
L. Ferroni, On the Ehrhart polynomial of minimal matroids, \textit{Discrete Comput. Geom.} (2021), 19pp.
\bibitem{Hamidoune}
 Y. O. Hamidoune, Counting certain pairings in arbitrary groups. \textit{Combin. Probab. Comput.} 20 (2011), no. 6, 855–865.
\bibitem{Hanley}
D. Hanley, J. L. Martin, D. McGinnis, D. Miyata,
G. D. Nasr, A. R. Vindas-Mel´endez, M. Yin, Ehrhart theory for paving
and panhandle matroids, \textit{ Adv. Geom.} 23 (2023), no. 4, 501–526.
\bibitem{Lev}
  V. F. Lev, The rectifiability threshold in abelian groups. \textit {Combinatorica} 28 (2008), no. 4, 491–497.
 \bibitem{Levi}
 F. W. Levi, Ordered groups, \textit{Proc. Indian Acad. Sci.}, 16 (1942), 256--263.
 \bibitem{Losonczy}
  J. Losonczy, On matchings in groups, \textit{Adv. in Appl. Math.} 20 (1998), no. 3, 385–391.
\bibitem{Oxley}
J. Oxley, \textit{Matroid theory}, second ed., Oxford Graduate Texts in Mathematics,
vol. 21, Oxford University Press, Oxford, 2011.
\bibitem{Wakeford}
E. K. Wakeford, On Canonical Forms. \textit{Proc. London Math. Soc.} (2) 18 (1920), 403–410.


 \end{thebibliography}
\end{document}